\documentclass[12pt,reqno]{amsart}

\usepackage[margin=1in]{geometry}  
\usepackage{amsmath}
\usepackage{amssymb}
\usepackage[pdftex]{graphicx}
\usepackage{verbatim}   
\usepackage{color}    
\usepackage{mathrsfs}  
\usepackage[colorlinks=true]{hyperref}
\hypersetup{urlcolor=blue, citecolor=blue, pdfstartview=FitH}
\hypersetup{pdfstartview={XYZ null null 1.25}}
\usepackage{fontenc}
\usepackage{textcomp}
\usepackage{fix-cm}
\usepackage{array}
\usepackage{enumerate}
\usepackage{latexsym, amsfonts, amssymb}
\usepackage{color}
\usepackage{txfonts}
\usepackage{comment}
\usepackage{subfigure}
\usepackage{accents}

\numberwithin{equation}{section}

\newtheorem{theorem}[equation]{Theorem}

\newtheorem{definition}[equation]{Definition}
\newtheorem{remark}[equation]{Remark}
\newtheorem{example}[equation]{Example}

\begin{document}

\title[Stagnation-point form solutions of 2D Euler]
{Regularity of Stagnation-point form solutions of the Two-dimensional Euler Equations}

\author[Alejandro Sarria]{Alejandro Sarria}
\address{Department of Mathematics \\
University of Colorado Boulder \\
Boulder, CO 80309-0395 USA \\
}
\email{alejandro.sarria@colorado.edu}

\thanks{The author carried out part of this work as a doctoral candidate at the University of New Orleans.}



\begin{abstract}
A class of semi-bounded solutions of the two-dimensional incompressible Euler equations satisfying either periodic or Dirichlet boundary conditions is examined. For smooth initial data, new blowup criteria in terms of the initial concavity profile is presented and the effects that the boundary conditions have on the global regularity of solutions is discussed. In particular, by deriving a formula for a general solution along Lagrangian trajectories, we describe how periodicity can prevent blow-up. This is as opposed to Dirichlet boundary conditions which, as we will show, allow for the formation of singularities in finite time. Lastly, regularity of solutions arising from non-smooth initial data is briefly discussed.
\end{abstract}

\maketitle

\section{Introduction}
\label{sec:intro}

We are concerned with regularity of solutions to the initial value problem
\begin{equation}
\label{eq:one}
\begin{cases}
u_{xt}+uu_{xx}-u_x^2=-2\int_0^1{u_x^2\,dx},\,\,\,\,\,\,\,\,\,&t>0,
\\
u(x,0)=u_0(x),\,\,\,\,\,\,\,\,\,\,\,\,\,\,&x\in[0,1],
\end{cases}
\end{equation}
with smooth initial data $u_0$ and either periodic
\begin{equation}
\label{eq:pbc}
\begin{split}
u(0,t)=u(1,t),\,\,\,\,\,\,\,\,\,\,\,\,u_x(0,t)=u_x(1,t),
\end{split}
\end{equation}
or Dirichlet boundary conditions
\begin{equation}
\label{eq:dbc}
\begin{split}
u(0,t)=u(1,t)=0.
\end{split}
\end{equation}
Equation (\ref{eq:one})i) was first derived in \cite{Proudman1} from the 2D incompressible Euler equations
\begin{equation}
\label{eq:2deuler}
\boldsymbol{u}_t+(\boldsymbol{u}\cdot\nabla)\boldsymbol{u}=-\nabla p,\,\,\,\,\,\,\,\,\,\,\,\,\,\,\,\nabla\cdot\boldsymbol u=0
\end{equation}
by introducing a stream function $\psi(x,y,t)=yu(x,t)$, resulting in velocity vectors of the form $\boldsymbol{u}(x,y,t)=(\psi_y,-\psi_x)=(u,-yu_x)$, also known as ``stagnation point-form'' velocity fields. Alternatively, (\ref{eq:one})i) may be obtained in the study of axisymmetric flows without swirl through the cylindrical coordinate representation $u^r=-yu_x(x,t)$,\, $u^x=u(x,t)$ (\cite{Weyl1}, \cite{Escher1}). Moreover, differentiating (\ref{eq:one})i) in space yields 
\begin{equation}
\label{eq:vorticitymodel}
\omega_{t}+u\omega_{x}=\omega u_x,\,\,\,\,\,\,\,\,\,\,\,\,\,\,\,\omega=u_{xx},
\end{equation}
an equation derived in \cite{Gregorio1} as a 1D model for the well-known 3D vorticity equation
$$\boldsymbol{\omega}_t+(\boldsymbol{u}\cdot\nabla)\boldsymbol{\omega}=(\boldsymbol{\omega}\cdot\nabla)\boldsymbol{u},\,\,\,\,\,\,\,\,\,\,\,\,\,\boldsymbol{\omega}=\nabla\times\boldsymbol{u}.$$ 
Recently, (\ref{eq:one})i) was obtained, and analyzed for a particular class of data, within the context of a reduced 2D model for the 3D inviscid primitive equations of large scale oceanic and atmospheric dynamics (\cite{Titi1}); see also \cite{Cao1} and \cite{Cao2} for the case of the viscous primitive equations and \cite{Wong1} for a blowup result related to \eqref{eq:one}i) in the setting of the hydrostatic Euler equations. In \cite{Childress}, the authors showed the existence of blowup solutions to (\ref{eq:one}) for a particular choice of smooth initial data satisfying Dirichlet boundary conditions. Furthermore, via separation of variables they constructed antisymmetric blowup solutions from non-smooth initial data; see \cite{Okamoto2} and \cite{Titi1} for further generalizations of this approach and other applications. In the periodic setting, piecewise global weak solutions to (\ref{eq:one}) were constructed in \cite{Saxton1}, while in \cite{Aconstantin1} (see also \cite{Wunsch1}), the authors established blow-up criteria for odd initial data in terms of the time-dependent supremum or infimum of $u_x$. Moreover, in \cite{Okamoto1} (see also \cite{Okamoto2}), the authors showed that boundedness of $u_{xx}$ in the $L^2$ norm leads to global solutions, a result which closely resembles the classical Beale-Kato-Majda \cite{beale} blow-up criterion for (\ref{eq:2deuler}). Lastly, in \cite{Sarria1} we proved global existence in time of solutions of (\ref{eq:one}) for a particular class of smooth, periodic initial data via a direct approach that involved the derivation of representation formulae for solutions to the problem. 

In this paper, for smooth initial conditions, we present new regularity criteria for solutions of (\ref{eq:one}) in terms of their initial concavity profile. Our main results are summarized in Theorems \ref{thm:blow} and \ref{thm:global} in \S\ref{sec:blow}. Briefly, we prove that $u_x$ blows up in finite time as long as $u_0$ satisfies (\ref{eq:dbc}) and $u_0''(\overline\alpha_i)\neq0$ for $\overline\alpha_i$, $1\leq i\leq n$, denoting the finite number of locations in $[0,1]$ where $u_0'$ attains its greatest value\footnote[1]{The Reader may refer to Remark \ref{infinite} for a brief discussion on one case where there are infinitely many $\overline\alpha_i\in[0,1]$.}. In contrast, if the smooth initial data is periodic, then solutions are shown to remain smooth for all time. More particularly, in the latter case we discuss how the order $k\geq1$ of the inflection point $\overline\alpha_i$ (see definition \ref{def:order}) determines the asymptotic behaviour of global solutions as $t\to+\infty$. Lastly, we briefly examine the behaviour of solutions arising from initial data $u_0$ that is, at least, $C^1[0,1]\,\, a.e.$ We remark that in this article, ``blow-up'' will refer to $u_x$ diverging in the $L^{\infty}[0,1]$ norm.

The outline of the paper is as follows. In \S\ref{sec:solution} we derive a representation formula for $u_x$ along Lagrangian paths. Using this general solution, regularity is then studied in \S\ref{sec:blow}, while specific examples are deferred to \S\ref{sec:examples}.\footnote[2]{We note that the formula derived in \S\ref{sec:solution} is a special case of the representation formulae established in \cite{Sarria1}, however, its derivation is presented here for the sake of completeness and convenience of the reader.}

\section{The Representation Formula}
\label{sec:solution}
We now give an outline for the derivation of a representation formula for $u_x$ along Lagrangian trajectories. For fixed $\alpha\in[0,1]$, define the flow of $u$, which we will denote by $\gamma$, via the IVP
\begin{equation}
\label{eq:charac}
\dot\gamma(\alpha,t)=u(\gamma(\alpha,t),t),\,\,\,\,\,\,\,\,\,\,\,\,\,\,\,\gamma(\alpha,0)=\alpha
\end{equation}
where $\cdot\equiv\frac{d}{dt}$. Since (\ref{eq:charac}) implies
\begin{equation}
\label{eq:jacid}
\begin{split}
\dot\gamma_{\alpha}=(u_x\circ\gamma)\cdot\gamma_{\alpha},
\end{split}
\end{equation}
we use (\ref{eq:one})i) and (\ref{eq:jacid}) to obtain
\begin{equation}
\label{eq:chain}
\begin{split}
\ddot\gamma_\alpha=
2\left((\gamma^{-1}_{\alpha}\cdot\dot\gamma_{\alpha})^2-\int_0^1{u_x^2\,dx}\right)\cdot\gamma_\alpha\,.
\end{split}
\end{equation}
Setting $I(t)=-2\int_0^1{u_x^2\,dx}$, then
\begin{equation}
\label{eq:jacid2}
\begin{split}
I(t)=\frac{\ddot\gamma_{\alpha}\cdot\gamma_{\alpha}-2\dot\gamma_{\alpha}^{\,2}}{\gamma_{\alpha}^{\,2}}=-\gamma_{\alpha}\cdot\left(\gamma_{\alpha}^{-1}\right)^{\ddot{}},
\end{split}
\end{equation}
and so
\begin{equation}
\label{eq:nonhomo2}
\begin{split}
\ddot\omega(\alpha,t)+I(t)\omega(\alpha,t)=0
\end{split}
\end{equation}
for
\begin{equation}
\label{eq:lin}
\begin{split}
\omega(\alpha,t)=\gamma_{\alpha}(\alpha,t)^{-1}.
\end{split}
\end{equation}
Let $\phi_1(t)$ and $\phi_2(t)$ be two linearly independent solutions to (\ref{eq:nonhomo2}) such that $\phi_1(0)=\dot{\phi}_2(0)=1$, $\dot{\phi}_1(0)=\phi_2(0)=0$. Then by Abel's formula,  $\text{W}(\phi_1,\phi_2)(t)\equiv1,\, t\geq 0$, where W$(g,h)$ denotes the wronskian of $g$ and $h.$ We look for solutions of (\ref{eq:nonhomo2}), satisfying appropriate initial data, of the form
\begin{equation}
\label{eq:unonhomosol2}
\begin{split}
\omega(\alpha,t)=c_1(\alpha)\phi_1(t)+c_2(\alpha)\phi_2(t)
\end{split}
\end{equation}
where $\phi_2(t)=\phi_1(t)\int_0^t{\phi_1^{-2}(s)\,ds}$, by reduction of order. Since $\dot{\omega}=-\gamma_\alpha^{-2}\dot{\gamma_\alpha}$ and $\gamma_\alpha(\alpha,0)=1,$ then $\omega(\alpha, 0)=1$ and $\dot{\omega}(\alpha, 0)=-u'_0(\alpha)$, from which $c_1(\alpha)$ and $c_2(\alpha)$ are obtained. Combining these results give
\begin{equation}
\label{eq:compat}
\begin{split}
\omega(\alpha,t)=\phi_1(t)\left(1-\eta(t)u_0'(\alpha)\right),\,\,\,\,\,\,\,\,\,\,\,\,\eta(t)=\int_0^t\frac{ds}{\phi_1^2(s)}.
\end{split}
\end{equation}
Now, (\ref{eq:lin}) and (\ref{eq:compat}) imply that
\begin{equation}
\label{eq:jaco}
\begin{split}
\gamma_{\alpha}(\alpha,t)=\left(\phi_1(t)\mathcal{J}(\alpha,t)\right)^{-1}, 
\end{split}
\end{equation}
where 
\begin{equation}
\label{eq:J}
\begin{split}
\mathcal{J}(\alpha,t)=1-\eta(t)u_0^\prime(\alpha),\,\,\,\,\,\,\,\,\,\,\,\,\,\,\,\,\mathcal{J}(\alpha,0)=1.
\end{split}
\end{equation}
However, uniqueness of solution to (\ref{eq:charac}) requires that, for as long as $u$ is defined,
\begin{equation}
\label{eq:percha}
\begin{split}
\gamma(\alpha+1,t)-\gamma(\alpha,t)\equiv1
\end{split}
\end{equation}
for periodic solutions, or 
\begin{equation}
\label{eq:percha2}
\begin{split}
\gamma(0,t)\equiv0,\,\,\,\,\,\,\,\gamma(1,t)\equiv1
\end{split}
\end{equation}
for Dirichlet boundary conditions. Either way, we have that\, $\int_0^1{\gamma_{\alpha}\,d\alpha}\equiv1$,\, so that spatially integrating (\ref{eq:jaco}) yields
$$\phi_1(t)=\int_0^1{\mathcal{J}(\alpha,t)^{-1}d\alpha}.$$
Consequently, by setting 
\begin{equation}
\label{eq:def}
\begin{split}
\mathcal{K}_i(\alpha, t)=\frac{1}{\mathcal{J}(\alpha,t)^{i+1}},\qquad\,\,\,\,\,\,\,\,\,\,\bar{\mathcal{K}}_i(t)=\int_0^1{\mathcal{K}_i(\alpha, t)\,d\alpha},
\end{split}
\end{equation}
for $i=0,1,...,n$, we can write $\gamma_\alpha$ in the form
\begin{equation}
\label{eq:sum}
\gamma_\alpha={\mathcal K}_0/{\bar{\mathcal K}}_0.
\end{equation}
Then using (\ref{eq:jacid}) and (\ref{eq:sum}) we obtain, after simplification, 
\begin{equation}
\label{eq:mainsolu}
\begin{split}
u_x(\gamma(\alpha,t),t)=(\ln({\mathcal{K}_0/\bar{\mathcal K}_0}))^{{}^.}=\frac{1}{\eta(t)\bar{\mathcal K}_0(t)^{2}}
\left(\frac{1}{\mathcal{J}(\alpha, t)}-\frac{\bar{\mathcal{K}}_1(t)}{\bar{\mathcal K}_0(t)}
\right).
\end{split}
\end{equation}
Moreover, differentiating (\ref{eq:compat})ii) gives
\begin{equation}
\label{eq:etaivp}
\begin{split}
\dot{\eta}(t)=\bar{\mathcal{K}}_0(t)^{-2},\,\,\,\,\,\,\,\,\,\,\,\eta(0)=0,
\end{split}
\end{equation}   
from which it follows that the existence of an eventual finite blow-up time $t_*>0$ will depend, in part, upon convergence of the integral
\begin{equation}
\label{eq:assympt}
\begin{split}
t(\eta)=\int_0^{\eta}{\left(\int_0^1{\frac{d\alpha}{1-\mu u_0^\prime(\alpha)}}\right)^{2}\,d\mu}
\end{split}
\end{equation}
as $\eta\uparrow\eta_*$ for $\eta_*>0$ to be defined. 
Finally, assuming sufficient smoothness, we may use (\ref{eq:sum}) and (\ref{eq:mainsolu}) to obtain 
\begin{equation}
\label{eq:preserv1}
\begin{split}
u_{xx}(\gamma(\alpha,t),t)=u_0^{\prime\prime}(\alpha)\cdot\gamma_{\alpha}(\alpha,t)=\frac{u_0^{\prime\prime}(\alpha)}{\mathcal{J}(\alpha,t)}\left(\int_0^1{\frac{d\alpha}{\mathcal{J}(\alpha,t)}}\right)^{-1}.
\end{split}
\end{equation}
Equation  (\ref{eq:preserv1}) implies that as long as a solution exists it will maintain its initial concavity profile. 

The reader may refer to \cite{Sarria1} for details on the above and formulae for $u(\gamma(\alpha,t),t)$. 
\begin{remark}
The representation formula (\ref{eq:mainsolu}) is a 1D analogue of a solution, derived by Constantin (\cite{Constantin1}), of the vertical component equation of the 3D incompressible Euler equations subject to an infinite energy, periodic class of solutions.
\end{remark}

\section{Global Estimates and Blow-up}
\label{sec:blow}
In this section we study the evolution of (\ref{eq:mainsolu}) from smooth initial data $u_0$. First, we introduce some terminology. For $\gamma$ as defined in (\ref{eq:charac}), set 
\begin{equation}
\label{eq:maxmin0}
\begin{split}
M(t)\equiv\sup_{\alpha\in[0,1]}u_x(\gamma(\alpha,t),t),\,\,\,\,\,\,\,\,\,\,\,\,M(0)=M_0,
\end{split}
\end{equation}
and
\begin{equation}
\label{eq:defeta*}
\begin{split}
\eta_*=\frac{1}{M_0}
\end{split}
\end{equation}
where $M_0>0$ denotes the greatest value attained by $u_0'$ at a finite number of locations $\overline\alpha_i\in[0,1]$, $1\leq i\leq n$, for some $n\in\mathbb{N}$.\footnote[3]{In \S\ref{sec:blow}, we remark on a case where $M_0$ is achieved at an infinite number of points.} Then (\ref{eq:mainsolu}) implies that
\begin{equation}
\label{eq:max}
\begin{split}
M(t)=u_x(\gamma(\overline\alpha_i,t),t)
\end{split}
\end{equation}
for all $1\leq i\leq n$ and for as long as solutions are defined\footnote[4]{A result similar to (\ref{eq:max}) follows for $m(t)\equiv\inf_{\alpha}u_x(\gamma,t)$ if we let $\underline\alpha_j$, $1\leq j\leq m$, denote the finitely many points where $u_0'$ attains its least value.}. Moreover, as $\eta\uparrow \eta_*$, the space-dependent term in (\ref{eq:mainsolu}) will diverge for certain choices of $\alpha$ and not at all for others. Specifically, $\mathcal{J}(\alpha,t)^{-1}$ will blowup  earliest as $\eta\uparrow \eta_*$ at $\alpha=\overline\alpha_i$,
\begin{equation*}
\begin{split}
\mathcal{J}(\overline\alpha_i,t)^{-1}=(1-\eta(t)M_0)^{-1}\to+\infty\,\,\,\,\,\,\,\text{as}\,\,\,\,\,\,\,\eta\uparrow\eta_*.
\end{split}
\end{equation*}
However, blow-up of $M(t)$ in (\ref{eq:max}) does not necessarily follow from this; we will need to estimate the behaviour, as $\eta\uparrow\eta_*$, of the time-dependent integrals $\bar{\mathcal{K}}_0(t)$ and $\bar{\mathcal{K}}_1(t)$. To do this, we use a Taylor expansion of $u_0'$ about $\overline\alpha_i$. Suppose 
\begin{equation}
\label{eq:expnew0}
u_0^\prime(\alpha)\sim M_0+C_1\left|\alpha-\overline\alpha_i\right|^q
\end{equation}
holds for all $1\leq i\leq n$, some $C_1\in\mathbb{R}^-$, $0\leq\left|\alpha-\overline\alpha_i\right|\leq r\leq1$ and either $q=1$, or $q=k+1$, $k\geq1$ odd. Note that the choice of $q$ in (\ref{eq:expnew0}) depends on the vanishing, or not, of $u_0''$ at $\overline\alpha_i\in[0,1]$, as well as on the corresponding set of boundary conditions. Due to the smoothness of $u_0$, there are two main possibilities to consider. First, notice that no smooth function $u_0'$ can attain its greatest value $M_0>0$ somewhere in the interior $(0,1)$ while satisfying (\ref{eq:expnew0}) for $q=1$. Indeed, suppose $q=1$ and assume there exists, say, $\overline\alpha_1\in(0,1)$. Then (\ref{eq:expnew0}) implies that $u_0''$ has jump discontinuities of finite magnitude at $\overline\alpha_1$. From this we conclude that if the data is smooth and $u_0'$ satisfies (\ref{eq:expnew0}) for $q=1$, then for all $1\leq i\leq n$, $\overline\alpha_i$ must be a boundary point. An example in the Dirichlet setting would be $u_0(\alpha)=\alpha(1-\alpha)$, which has $\overline\alpha_1=0$ and $M_0=1$.\footnote[5]{Initial data similar to this was used in \cite{Childress} to construct a blowup solution.} In fact, of the boundary conditions (\ref{eq:pbc}) and (\ref{eq:dbc}), only the latter allows for smooth data satisfying such conditions. Indeed, suppose a periodic function $u_0$ satisfies (\ref{eq:expnew0}) for $q=1$ and $M_0=u_0'(0)=u_0'(1)>u_0'(\alpha)$ for all $\alpha\in(0,1)$. Then by periodicity of $u_0$ and the definition of $\overline\alpha_i$, we have that $0>u_0''(0)=u_0''(1)$. But using (\ref{eq:expnew0}) with $q=1$ gives 
$$0>u_0''(1)=\lim_{h\to0^{^{-}}}{\frac{u_0'(1+h)-M_0}{h}}\sim\lim_{h\to0^{^{-}}}{\frac{(M_0+\left|C_1\right|h)-M_0}{h}}=\left|C_1\right|,$$
a contradiction. We conclude that if the data is periodic and satisfies (\ref{eq:expnew0}) for $q=1$, then $\overline\alpha_i\in(0,1)$ and, thus, $u_0$ cannot be smooth due to our previous argument. From the above discussion, it follows that regularity of (\ref{eq:mainsolu}) with smooth initial data can be examined by considering two different cases, each characterized by the vanishing, or not, of $u_0''$ at $\overline\alpha_i$. For simplicity, we will assume that the local profile of $u_0''$ near all $\overline\alpha_i$ is the same. However, how to proceed if this is not the case will be clear from the subsequent arguments. Lastly, the structure of \eqref{eq:preserv1} implies that in the case of blowup in $u_x$, divergence in third or higher order spatial derivatives of $u$ can only occur as $\eta$ approaches $\eta_*$. This is easily verified by differentiating \eqref{eq:preserv1} with respect to $\alpha$ and using \eqref{eq:sum}.

In Theorem \ref{thm:blow} below, we show that finite-time blowup in $u_x$ from smooth initial data can occur under Dirichlet boundary conditions as long as $u_0''(\overline\alpha_i)\neq0$ for all $i$. In contrast, in Theorem \ref{thm:global} we prove that solutions subject to (\ref{eq:pbc}), and/or (\ref{eq:dbc}), will persist globally in time if $u_0''$ vanishes at $\overline\alpha_i$ for at least one $i$. More particularly, if $u_0''(\overline\alpha_i)=0$ for all $i$, the latter scenario will imply that
\begin{equation}
\label{eq:expnew0order}
u_0'(\alpha)\sim M_0+C_1\left|\alpha-\overline\alpha_i\,\right|^{k+1}
\end{equation}
for all $0\leq\left|\alpha-\overline\alpha_i\right|\leq r$, with $r>0$ defined by\, $r\equiv\min_{1\leq i\leq n}\{r_i\},$
and where each $r_i>0$ corresponds to at least one $\overline\alpha_i$. In \eqref{eq:expnew0order}, $k\geq1$ is odd and fixed, and represents the order of $\overline\alpha_i$ (see definition below)\footnote[6]{Notice that $k\geq1$ must be odd due to $u_0'$ being even in a neighbourhood of $\overline\alpha_i$.}, while
\begin{equation}
\label{eq:c1}
C_1=\frac{u_0^{^{(k+2)}}(\overline\alpha_i)}{(k+1)!}<0.
\end{equation}
Even though $C_1$ may vary from one $\overline\alpha_i$ to the next, what will matter to us while deriving upcoming estimates is that the negativity of these constants is independent of the particular location $\overline\alpha_i$. Consequently, there will be no need for us to differentiate among these constants neither qualitatively nor notationally. Lastly, we point out that for a solution to be global in time, it will suffice that $u_0''$ be zero at, at least, one $\overline\alpha_i$. Essentially, what happens is that if $u_0''(\overline\alpha_1)=0$ for some $\overline\alpha_1\in[0,1]$, but there is also $\overline\alpha_2\in[0,1]$ such that $u_0''(\overline\alpha_2)\neq0$, then it will become clear in the next section that the local profile of $u_0''$ near $\overline\alpha_1$ dominates and determines the behavior of the integral terms. This is precisely why blowup requires  $u_0''(\overline\alpha_i)\neq0$ for all $i$.

\begin{definition}
\label{def:order}
Suppose a smooth function $f(x)$ satisfies $f(x_0)=0$ for $f$ not identically zero. We say $f$ has a zero of order $k\in\mathbb{Z}^+$ at $x=x_0$ if 
$$f(x_0)=f'(x_0)=...=f^{(k-1)}(x_0)=0,\,\,\,\,\,\,\,\,\,\,\,\,f^{(k)}(x_0)\neq0.$$
\end{definition}
We begin by establishing Theorem \ref{thm:blow} below, which provides new criteria for the existence of finite-time blow-up solutions of (\ref{eq:one}).

\begin{theorem}
\label{thm:blow}
Consider the initial value problem (\ref{eq:one}) for smooth initial data $u_0(x)$ satisfying (\ref{eq:dbc}) and such that $u_0'(x)$ attains its greatest value $M_0>0$ at a finite number of locations $\overline\alpha_i\in[0,1]$, $1\leq i\leq n$. If\, $u_0''(\overline\alpha_i)\neq0$ for all $1\leq i\leq n$, then there exists a finite $t_*>0$ such that the smooth solution $u_x$ of (\ref{eq:one}) blows up as $t\uparrow t_*$. More particularly, the maximum $M(t)=u_x(\overline\alpha_i,t)\to+\infty$ as $t\uparrow t_*$, while, for $\alpha\neq\overline\alpha_i$, $u_x(\gamma(\alpha,t),t)\to-\infty$.
\end{theorem}
\begin{proof}
From our previous discussion, smoothness of $u_0$ implies that $\overline\alpha_i\in\{0,1\}$. Without loss of generality, and to keep the presentation as simple as possible, we will assume that $u_0'$ attains its greatest value $M_0>0$ only at $\overline\alpha=0$ and $u_0''(0)\neq0$. Then via a Taylor expansion this implies that for $\alpha>0$ small, 
\begin{equation}
\label{eq:expnew0dbc}
u_0^\prime(\alpha)\sim M_0+C_1\alpha,\,\,\,\,\,\,\,\,\,\,\,\,\,\,C_1=u_0''(0)<0,
\end{equation}
and so for $\epsilon>0$ small, there is $0<r\leq1$ such that
\begin{equation}
\label{eq:againdiri}
\begin{split}
\epsilon+M_0-u_0'(\alpha)\sim\epsilon+\left|C_1\right|\alpha
\end{split}
\end{equation}
for $0\leq\alpha\leq r$. Consequently
\begin{equation}
\label{eq:app}
\begin{split}
\int_{0}^{r}{\frac{d\alpha}{\epsilon+M_0-u_0'(\alpha)}}\sim\int_{0}^{r}{\frac{d\alpha}{\epsilon+\left|C_1\right|\alpha}}=-\frac{1}{\left|C_1\right|}\ln\epsilon.
\end{split}
\end{equation}
Setting $\epsilon=\frac{1}{\eta}-M_0$ into (\ref{eq:app}) then implies that, for $\eta_*-\eta>0$ small,
\begin{equation}
\label{eq:intest1}
\begin{split}
\bar{\mathcal{K}}_0(t)\sim-\frac{M_0}{\left|C_1\right|}\ln(\eta_*-\eta).
\end{split}
\end{equation}
In a similar fashion, the second integral can be shown to diverge at a rate
\begin{equation}
\label{eq:intest2}
\begin{split}
\bar{\mathcal{K}}_1(t)\sim\frac{1}{\left|C_1\right|}(\eta_*-\eta)^{-1}.
\end{split}
\end{equation}
For $\alpha=\overline\alpha=0$, the above integral estimates, along with (\ref{eq:percha2})i) and (\ref{eq:max}), imply that the space-dependent term in (\ref{eq:mainsolu}) dominates and, as a result, the maximum $M(t)=u_x(0,t)$ satisfies
\begin{equation}
\label{eq:maxblow}
\begin{split}
M(t)\sim\left(\frac{C_1}{M_0}\right)^2\left(\frac{1}{(\eta_*-\eta)\ln^2(\eta_*-\eta)}\right)
\end{split}
\end{equation}
for $\eta_*-\eta>0$ small. Consequently
\begin{equation}
\label{eq:maxblow2}
\begin{split}
M(t)\to+\infty
\end{split}
\end{equation}
as $\eta\uparrow\eta_*$. In contrast, for $\alpha\neq0$ and $0\leq\eta\leq\eta_*$, the space-dependent term now remains finite and the second term dominates. This implies that
\begin{equation}
\label{eq:maxblow3}
\begin{split}
u_x(\gamma(\alpha,t),t)\sim\left(\frac{C_1}{M_0}\right)^2\left(\frac{1}{(\eta_*-\eta)\ln^3(\eta_*-\eta)}\right)\to-\infty
\end{split}
\end{equation}
as $\eta\uparrow\eta_*$. Lastly, the existence of a finite blowup time $t_*>0$ follows from using (\ref{eq:intest1}) on (\ref{eq:etaivp}), which yields 
\begin{equation}
\label{eq:time1}
\begin{split}
t_*-t\sim\left(\frac{M_0}{C_1}\right)^2\int_{\eta}^{\eta_*}{\ln^2(\eta_*-\mu)\,d\mu}.
\end{split}
\end{equation}
Then for $\eta_*-\eta>0$ small and $C=2(M_0/C_1)^2$, the above gives, after simplification, the asymptotic relation
\begin{equation}
\label{eq:timefin}
\begin{split}
t_*-t\sim C\,(\eta_*-\eta).
\end{split}
\end{equation}
\end{proof}
\begin{remark}
It will be clear from the estimates in the proof of the next Theorem that $u_0''(\overline\alpha_i)\neq0$ for all $i$ is indeed necessary for finite-time blowup from smooth initial conditions.
\end{remark}
Our next result examines global existence of solutions.

\begin{theorem}
\label{thm:global}
Consider the initial value problem (\ref{eq:one}) for smooth initial data $u_0(x)$ satisfying (\ref{eq:pbc}) (and/or (\ref{eq:dbc})) and such that $u_0'(x)$ attains its greatest value $M_0>0$ at a finite number of locations $\overline\alpha_i\in[0,1]$, $1\leq i\leq n$. If $u_0''$ vanishes at $\overline\alpha_i$ for at least one $i$, then solutions remain smooth for all time. In particular, a solution will stay smooth for all time if $u_0''$ has a zero of order $k\geq1$ (see definition \ref{def:order}) at every $\overline\alpha_i$\,, with $u_x$ converging to a non-trivial steady-state as $t\to+\infty$ if $k=1$, but vanishing when $k>1$. 
\end{theorem}
\begin{proof}
First suppose $u_0''$ has a zero of order $k\geq1$ at all $\overline\alpha_i$. Then smoothness of $u_0$ implies that in a neighbourhood of those $\overline\alpha_i\in(0,1)$, $u_0'$ satisfies (\ref{eq:expnew0order}) for fixed $k\geq1$ odd. Similarly for the case where there are $\overline\alpha_i\in\{0,1\}$ due to our assumption $u_0''(\overline\alpha_i)=0$ and definition \ref{def:order}. In order to simplify subsequent computations, assume that $u_0'$ attains its maximum $M_0>0$ only at one location $\overline\alpha\in(0,1)$.\footnote[7]{A slightly modification of the argument presented below will suffice to accommodate the case of finitely many $\overline\alpha_i\in[0,1]$.} Then from (\ref{eq:expnew0order}), there is $r>0$ and fixed $k\geq1$ odd, such that
\begin{equation}
\label{eq:againdiri2}
\begin{split}
\epsilon+M_0-u_0'(\alpha)\sim\epsilon-C_1\left|\alpha-\overline\alpha\,\right|^{k+1}
\end{split}
\end{equation}
for $\epsilon>0$ small, $0\leq\left|\alpha-\overline\alpha\right|\leq r$ and constant $C_1<0$ as in (\ref{eq:c1}). Letting $b\in\{1,2\}$ and using the above we have that
\begin{equation*}
\begin{split}
\int_{\overline\alpha-r}^{\overline\alpha+r}{\frac{d\alpha}{(\epsilon+M_0-u_0'(\alpha))^b}}&\sim\int_{\overline\alpha-r}^{\overline\alpha+r}{\frac{d\alpha}{(\epsilon-C_1\left|\alpha-\overline\alpha\,\right|^{k+1})^b}}
\\
&=\epsilon^{-b}\left\{\int_{\overline\alpha-r}^{\overline\alpha}{\left(1+\frac{\left|C_1\right|}{\epsilon}\left(\overline\alpha-\alpha\right)^{1+k}\right)^{-b}d\alpha}+\int_{\overline\alpha}^{\overline\alpha+r}{\left(1+\frac{\left|C_1\right|}{\epsilon}\left(\alpha-\overline\alpha\right)^{1+k}\right)^{-b}d\alpha}\right\}.
\end{split}
\end{equation*}
Making the change of variables
$$\sqrt{\frac{\left|C_1\right|}{\epsilon}}\,(\overline\alpha-\alpha)^{\frac{k+1}{2}}=\tan\theta,\,\,\,\,\,\,\,\,\,\,\,\,\,\,\sqrt{\frac{\left|C_1\right|}{\epsilon}}\,(\alpha-\overline\alpha)^{\frac{k+1}{2}}=\tan\theta$$
in the first and respectively second integrals inside the braces, yields
\begin{equation}
\label{eq:general2}
\begin{split}
\int_{\overline\alpha-r}^{\overline\alpha+r}{\frac{d\alpha}{(\epsilon+M_0-u_0'(\alpha))^b}}\sim\frac{4\, \epsilon^{\frac{1}{1+k}-b}}{(1+k)\left|C_1\right|^{\frac{1}{1+k}}}\int_0^{\frac{\pi}{2}}{\frac{(\cos\theta)^{^{2b-\frac{k+3}{k+1}}}}{(\sin\theta)^{^{\frac{k-1}{k+1}}}}d\theta}
\end{split}
\end{equation}
for $\epsilon>0$ small and where $\frac{1}{1+k}-b<0$ for, particularly, $b\in\{1,2\}$ and $k\geq1$ odd. Then setting $\epsilon=\frac{1}{\eta}-M_0$ in (\ref{eq:general2}) gives
\begin{equation}
\label{eq:general3}
\begin{split}
\int_{0}^{1}{\frac{d\alpha}{\mathcal{J}(\alpha,t)^b}}\sim C(\eta_*-\eta)^{\frac{1}{1+k}-b}
\end{split}
\end{equation}
for $\eta_*-\eta>0$ small, $\eta_*=\frac{1}{M_0}$ and 
\begin{equation}
\label{eq:generalcst0}
\begin{split}
C=\frac{4\,M_0^{\frac{2}{1+k}-b}}{(1+k)\left|C_1\right|^{\frac{1}{1+k}}}\int_0^{\frac{\pi}{2}}{\frac{(\cos\theta)^{^{2b-\frac{k+3}{k+1}}}}{(\sin\theta)^{^{\frac{k-1}{k+1}}}}d\theta}.
\end{split}
\end{equation}
Note that (\ref{eq:generalcst0}) above is finite and positive. Indeed, since the gamma function satisfies (see for instance \cite{Gamelin1}) 
\begin{equation}
\label{eq:gammarel}
\begin{split}
\int_0^1{t^{p-1}(1-t)^{s-1}dt}=\frac{\Gamma(p)\Gamma(s)}{\Gamma(p+s)},\,\,\,\,\,\,\,\,\,\,\,\,\,\,\,\,\,\,\Gamma(1+y)=y\Gamma(y)
\end{split}
\end{equation}
for $p,s,y>0$, we can let $t=\sin^2\theta$, $p=\frac{1}{1+k}$ and $s=b-\frac{1}{1+k}$ into (\ref{eq:gammarel})i) to rewrite (\ref{eq:generalcst0}) as
\begin{equation}
\label{eq:generalcst}
\begin{split}
C=\frac{2\,M_0^{\frac{2}{1+k}-b}}{(1+k)\left|C_1\right|^{\frac{1}{1+k}}}\Gamma\left(\frac{1}{1+k}\right)\Gamma\left(b-\frac{1}{1+k}\right),
\end{split}
\end{equation}
which is clearly positive and well-defined. Now, using (\ref{eq:general3}) and (\ref{eq:generalcst}), we obtain the blow-up rates
\begin{equation}
\label{eq:k0}
\begin{split}
\bar{\mathcal{K}}_0(t)\sim C_2(\eta_*-\eta)^{-\frac{k}{1+k}}
\end{split}
\end{equation}
for
\begin{equation}
\label{eq:c2}
\begin{split}
C_2=\left(\frac{2}{1+k}\right)\Gamma\left(\frac{1}{1+k}\right)\Gamma\left(\frac{k}{1+k}\right)\left(\frac{M_0^{1-k}}{\left|C_1\right|}\right)^{\frac{1}{1+k}},
\end{split}
\end{equation}
and 
\begin{equation}
\label{eq:k1}
\begin{split}
\bar{\mathcal{K}}_1(t)\sim C_3(\eta_*-\eta)^{-\frac{1+2k}{1+k}}
\end{split}
\end{equation}
with
\begin{equation}
\label{eq:c3}
\begin{split}
C_3=\left(\frac{2}{1+k}\right)\Gamma\left(\frac{1}{1+k}\right)\Gamma\left(\frac{1+2k}{1+k}\right)\left(\frac{M_0^{-2k}}{\left|C_1\right|}\right)^{\frac{1}{1+k}}.
\end{split}
\end{equation}
Before studying the behaviour of (\ref{eq:mainsolu}) via the above estimates, it is important to note that (\ref{eq:gammarel})ii) implies  
\begin{equation}
\label{eq:c3/c2}
\begin{split}
\frac{C_3}{C_2}=\frac{1}{M_0}\left(\frac{k}{1+k}\right).
\end{split}
\end{equation}
Letting $\alpha=\overline\alpha$ in (\ref{eq:mainsolu}) and using (\ref{eq:max}), along with (\ref{eq:k0})-(\ref{eq:c3/c2}), we find that the maximum $M(t)=u_x(\gamma(\overline\alpha,t),t)$ satisfies
$$M(t)\sim\left(\frac{C_2^{^{-2}}}{1+k}\right)(\eta_*-\eta)^{\frac{k-1}{k+1}}$$
for $\eta_*-\eta>0$ small. Therefore
\begin{equation}
\label{eq:max0}
M(t)\to
\begin{cases}
\frac{\left|u_0'''(\overline\alpha)\right|}{(2\pi)^2},\,\,\,\,\,\,\,\,\,\,\,\,\,&k=1,
\\
0^+,\,\,\,\,\,\,\,\,\,\,\,\,\,\,\,\,&k=3,5,7,...
\end{cases}
\end{equation}
as $\eta\uparrow\eta_*$. For $\alpha\neq\overline\alpha$, the space-dependent term in (\ref{eq:mainsolu}) remains finite for $0\leq\eta\leq\eta_*$ and, as a result, an argument similar to the one above leads to 
$$u_x(\gamma(\alpha,t),t)\sim-\left(\frac{k}{(1+k)C_2^{^{2}}}\right)(\eta_*-\eta)^{\frac{k-1}{k+1}}$$
for $\eta_*-\eta>0$ small. Consequently, for $\alpha\neq\overline\alpha$, 
\begin{equation}
\label{eq:ux0}
u_x(\gamma(\alpha,t),t)\to
\begin{cases}
-\frac{\left|u_0'''(\overline\alpha)\right|}{(2\pi)^2},\,\,\,\,\,\,\,\,\,\,\,\,\,&k=1,
\\
0^-,\,\,\,\,\,\,\,\,\,\,\,\,\,\,\,\,&k=3,5,7,...
\end{cases}
\end{equation}
as $\eta\uparrow\eta_*$. Moreover, using (\ref{eq:k0}) on (\ref{eq:etaivp}) yields
\begin{equation}
\label{eq:time11}
\begin{split}
t_*-t\sim C_2^2\int_{\eta}^{\eta_*}{(\eta_*-\mu)^{-\frac{2k}{1+k}}d\mu}
\end{split}
\end{equation}
which, particularly for $k\geq1$ odd, implies 
\begin{equation}
\label{eq:timeinf}
\begin{split}
t_*=+\infty.
\end{split}
\end{equation}
Lastly, it can be easily shown by differentiating \eqref{eq:preserv1} with respect to $\alpha$ and then using \eqref{eq:sum}, that all higher-order spatial derivatives of $u_x$ will remain finite and continuous on $[0,1]$ for all $t\in\mathbb{R}^+$. For instance, the third-order derivative of $u$ along $\gamma$ is given by the simple formula
$$u_{xxx}(\gamma(\alpha,t),t)=u_0'''(\alpha)+\eta(t)\frac{u_0''(\alpha)^2}{\mathcal{J}(\alpha,t)}.$$
From this, we see that when $\alpha=\overline\alpha$,
$$u_{xxx}(\gamma(\overline\alpha,t),t)=u_0'''(\overline\alpha),$$
whereas, for $\alpha\neq\overline\alpha$,\,   $u_{xxx}$ stays finite for all $t>0$ since, for such choice of $\alpha$, the definition of $\overline\alpha$ and \eqref{eq:timeinf} imply that $\mathcal{J}(\alpha,t)>0$ for all $0\leq t\leq+\infty$.

Lastly, from the integral estimates \eqref{eq:k0} and \eqref{eq:k1}, and their counter-parts \eqref{eq:intest1} and \eqref{eq:intest2}, it is easy to see that if there are $1\leq l, m\leq n$, $l\neq m$, such that  $u_0''(\overline\alpha_l)\neq0$ but $u_0''(\overline\alpha_m)=0$, then the local behavior of $u_0''$ near $\overline\alpha_m$ dominates in the integral terms. This is why for a solution to be global, vanishing of $u_0''$ at $\overline\alpha_i$ for at least one $\alpha_i$ is sufficient. 
\end{proof}
\begin{remark}
\label{infinite}
Note that letting $q\to+\infty$ in (\ref{eq:expnew0}) implies that $u_0'(\alpha)\sim M_0$ in a neighbourhood of each $\overline\alpha_i$. Therefore, by letting $k\to+\infty$ above we can study regularity of solutions arising from initial data for which $u_0'$ attains its greatest value $M_0>0$ at an infinite number of locations in $[0,1]$. We find that, if solutions exist locally in time, they will persist for all time\footnote[8]{The reader may refer to \cite{Sarria1} for details on a related periodic problem.}. Also, by using a slightly extended argument (\cite{Sarria2}), we can examine regularity of solutions with initial data $u_0'$ that is, at least, $C^0(0,1)\, a.e.$ and satisfies (\ref{eq:expnew0}) for fixed $q\in\mathbb{R}^+$. In this case, our results indicate finite-time blowup in $u_x$ for $0<q<2$, but global existence in time if $q\geq2$; with $q=2$ a ``threshold'' value as it separates solutions vanishing as $t\to+\infty$ from those diverging at a finite time\footnote[9]{For $q=2$, $u_x$ converges to a \emph{non-trivial} steady-state as $t\to+\infty$.}.
\end{remark}

\section{Examples}
\label{sec:examples}

\begin{example}
Let $u_0(\alpha)=\alpha(\alpha-1)(\alpha-1/2)$, so that $u_0'(\alpha)=3\alpha^2-3\alpha+\frac{1}{2}$ achieves its greatest value $M_0=1/2$ at $\overline\alpha_{i}=\{0,1\}$, $i=1, 2$, and $u_0''(\overline\alpha_i)\neq0$ for every $i$. Also $\eta_*=2$ and, since
$$u_0'(\alpha)\sim M_0-3\alpha\qquad\quad\text{and}\qquad\quad u_0'(\alpha)\sim M_0-3\left|\alpha-1\right|$$
for $\alpha>0$ and respectively $1-\alpha>0$ small, then $u_0'$ satisfies (\ref{eq:expnew0}) for $q=1$. The integrals in (\ref{eq:mainsolu})i) evaluate to
\begin{equation}
\label{eq:k0ex1}
\begin{split}
\bar{\mathcal{K}}_0(t)=\frac{2\,\text{arctanh}(y(t))}{\sqrt{3\eta(t)(4+\eta(t))}},\,\,\,\,\,\,\,\,\,\,\,\,\,\,\,\,\,\int_0^1{\frac{u_0'(\alpha)}{\mathcal{J}(\alpha,t)^2}d\alpha}=\frac{d}{d\eta}\bar{\mathcal{K}}_0
\end{split}
\end{equation}
for $0\leq\eta<2$ and
$$y(t)=\frac{\sqrt{3\eta(t)(4+\eta(t))}}{2(1+\eta(t))}\,.$$
Using the above on (\ref{eq:mainsolu})i) and recalling (\ref{eq:percha}), we find that $M(t)=u_x(0,t)=u_x(1,t)\to+\infty$
as $\eta\uparrow2$ while, for $x\in(0,1)$,
$u_x(x,t)\to-\infty$. Finally (\ref{eq:etaivp}) and (\ref{eq:k0ex1})i) give a finite-time blowup $t_*\sim2.8$. See figure \ref{fig:fig}$(A)$ below.
\end{example}

\begin{example}
Let $u_0(\alpha)=\frac{1}{2\pi}\sin(2\pi\alpha)$. As in Example 1, $u_0'(\alpha)=\cos(2\pi\alpha)$ attains its greatest value $M_0=1$ at $\overline\alpha_i=\{0,1\}$, $i=1, 2$, while $\eta_*=1$; however, in this case $u_0''(\alpha)=-2\pi\sin(2\pi\alpha)$ vanishes at each $\overline\alpha_i$. Then for $\alpha>0$ and $1-\alpha>0$ small, $u_0'$ satisfies (\ref{eq:expnew0order}) for $k=1$ near each $\overline\alpha_i$, namely, both boundary points are zeros of $u_0''$ of order $k=1$. According to Theorem \ref{thm:global}, our solution will persist for all time. Indeed, the integrals in (\ref{eq:mainsolu})i) evaluate to
\begin{equation}
\label{eq:ge}
\begin{split}
\bar{\mathcal{K}}_0(t)=\frac{1}{\sqrt{1-\eta(t)^2}},\,\,\,\,\,\,\,\,\,\,\,\,\,\,\int_0^1{\frac{u_0^\prime(\alpha)}{\mathcal{J}(\alpha,t)^2}d\alpha}=\frac{d}{d\eta}\bar{\mathcal{K}}_0,
\end{split}
\end{equation}
both of which diverge to $+\infty$ as $\eta\uparrow\eta_*=1$. Moreover, (\ref{eq:ge}) and (\ref{eq:etaivp}) imply that $\eta(t)=\text{tanh}\,t,$\,
which we use on (\ref{eq:mainsolu})i), along with (\ref{eq:ge}), to obtain
\begin{equation}
\label{eq:uxex2}
\begin{split}
u_x(\gamma(\alpha,t),t)=\frac{\text{tanh}\,t-\cos(2\pi\alpha)}{\text{tanh}\,t\cos(2\pi\alpha)-1}.
\end{split}
\end{equation}
Clearly $$M(t)=u_x(\gamma(\overline\alpha_i,t),t)\equiv1\qquad\text{and}\qquad m(t)\equiv-1$$
for all $t\geq0$ and where\, $m(t)=\inf_{\alpha\in[0,1]}u_x(\gamma(\alpha,t),t).$ Further, for $\alpha\neq\overline\alpha_i$,
$$u_x(\gamma(\alpha,t),t)\to-1$$
as $\eta\uparrow 1.$ Finally, the above formula for $\eta$ implies that $t_*=\lim_{\eta\uparrow 1}{\text{arctanh}\,\eta}=+\infty.$
It is easy to see from (\ref{eq:ge}) and the formulas in \S\ref{sec:solution} that, in this case, the nonlocal term in \eqref{eq:one} remains constant; more particularly, $\int_0^1{u_x^2dx}\equiv1/2$. See figure \ref{fig:fig}($B$) below.
\end{example}

\begin{center}
\begin{figure}[!ht]
\includegraphics[scale=0.5]{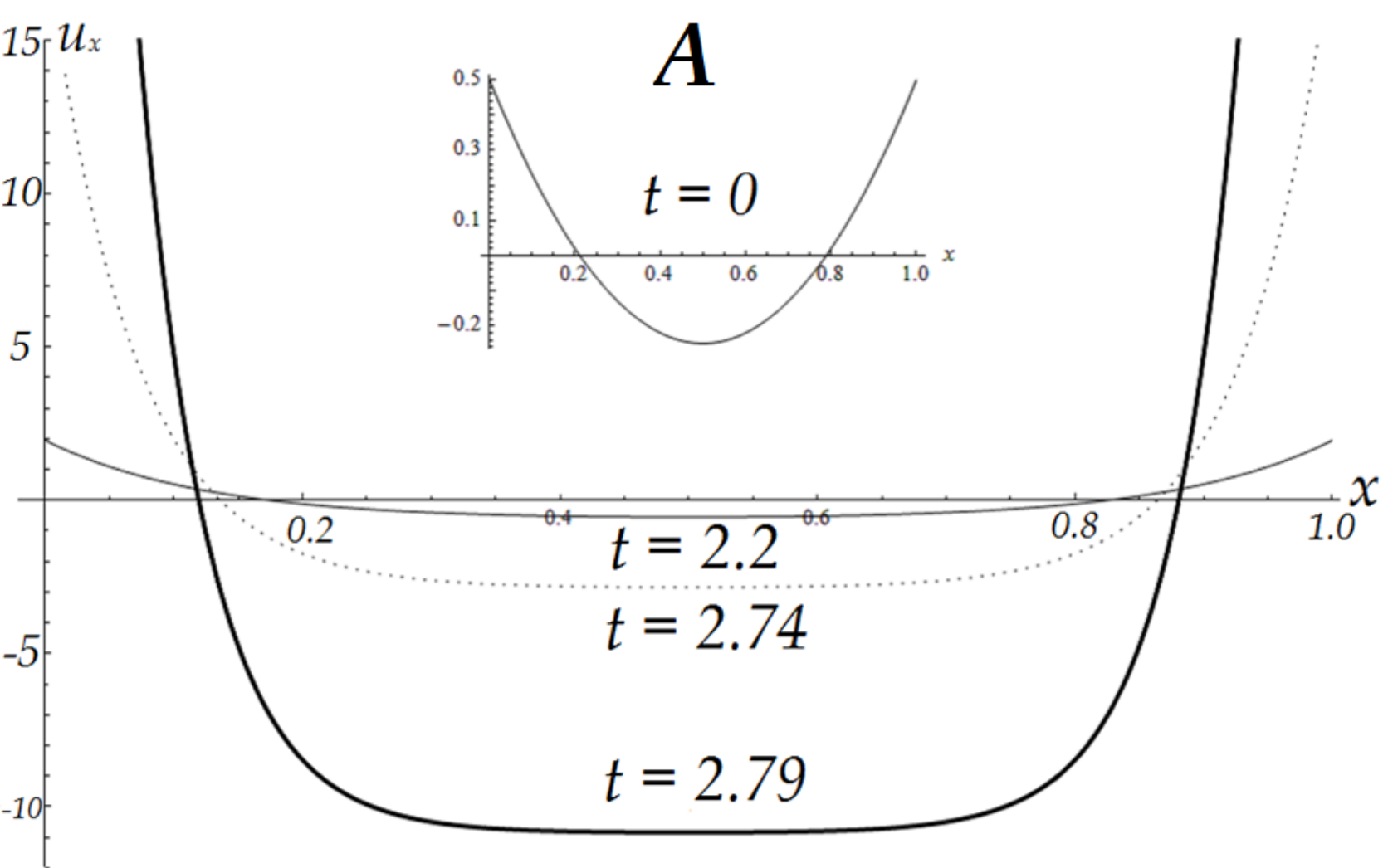} 
\includegraphics[scale=0.44]{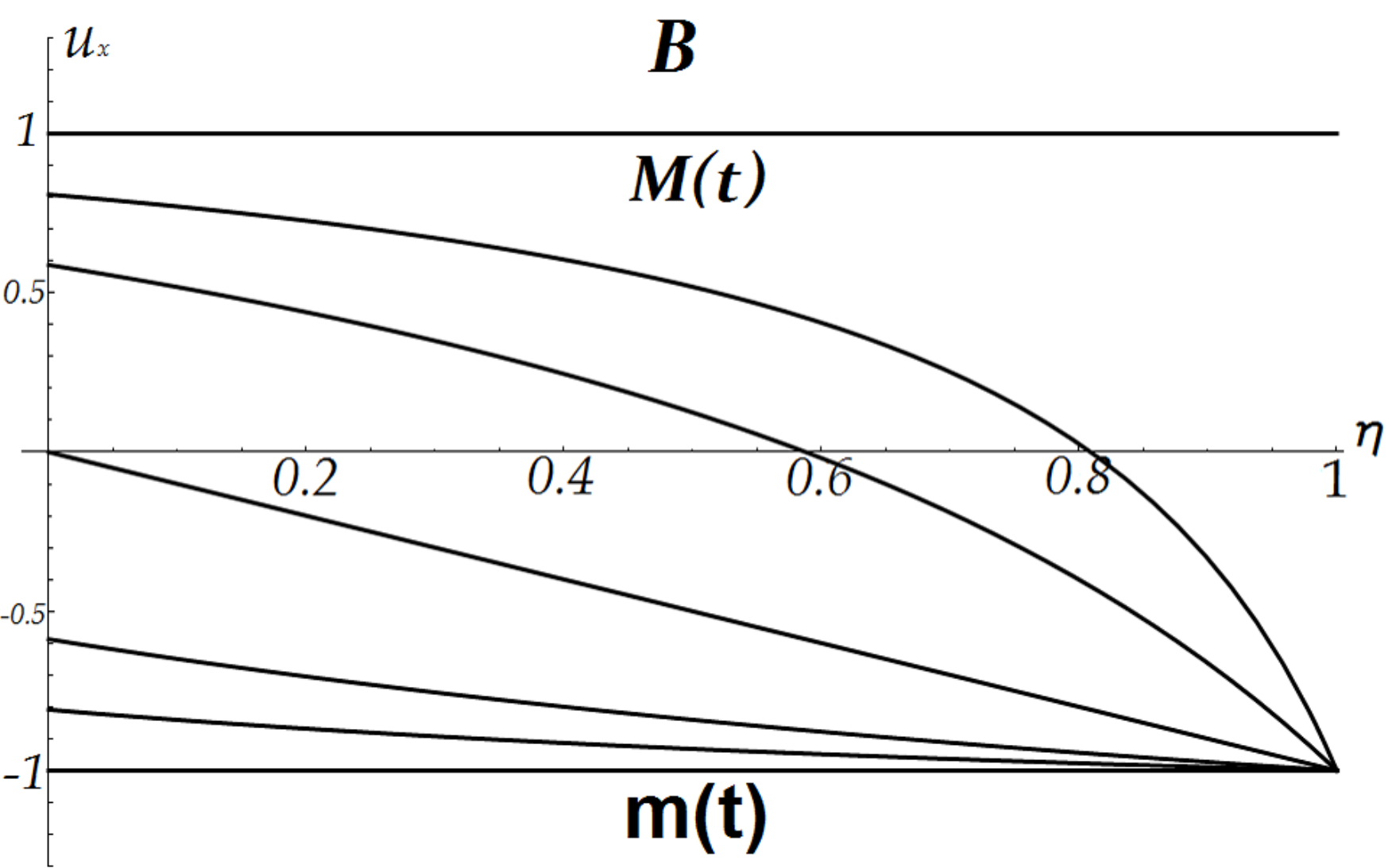} 
\caption{Figure $\textbf{A}$ for Example 1 depicts finite-time blowup $u_x(0,t)=u_x(1,t)\to+\infty$ as $t\uparrow t_*\sim2.8$ and $u_x(x,t)\to-\infty$ for $x\in(0,1)$. Figure $\textbf{B}$ for Example 2 represents the global solution $u_x\circ\gamma$ in (\ref{eq:uxex2}) as $t\to+\infty$. In this case $M(t)\equiv1$ and $m(t)\equiv-1$, whereas, for $\alpha\notin\{0,1/2,1\}$, $u_x\circ\gamma\to-1$.}
\label{fig:fig}
\end{figure}
\end{center}

\end{document}